%% file: basechange.tex
\documentclass[a4paper]{amsart}
\usepackage[leqno]{amsmath}
\usepackage{amssymb}
\usepackage{amscd}
\usepackage{amsthm}
\usepackage{mathtools}
\usepackage{bbm}
\usepackage{color}
\usepackage{enumerate}
\usepackage{cite}
\usepackage[utf8]{inputenc}
\usepackage[all,cmtip]{xy}
\usepackage{etoolbox}
\usepackage{tikz}
\usetikzlibrary{arrows}
\usepackage{extarrows}

\numberwithin{equation}{section}

\newcommand{\Z}{\ensuremath{\mathbb{Z}}}
\newcommand{\Q}{\ensuremath{\mathbb{Q}}}

\newcommand{\C}{\ensuremath{\mathbb{C}}}

\newcommand{\A}{\ensuremath{\mathbb{A}}}
\newcommand{\I}{\ensuremath{\mathbb{I}}}
\newcommand{\OO}{\ensuremath{\mathcal{O}}}

\DeclareMathOperator{\Hom}{Hom}

\DeclareMathOperator{\id}{id}

\DeclareMathOperator{\GL}{\mathbb{GL}}
\DeclareMathOperator{\PGL}{\mathbb{PGL}}

\DeclareMathOperator{\GG}{\mathbb{G}}
\DeclareMathOperator{\PGG}{\mathbb{PG}}

\DeclareMathOperator{\ord}{ord}

\DeclareMathOperator{\St}{St}

\DeclareMathOperator{\cyc}{cyc}
\DeclareMathOperator{\anti}{anti}
\DeclareMathOperator{\HH}{H}
\DeclareMathOperator{\MS}{\mathcal{M}}

\newcommand{\PP}{\ensuremath{\mathbb{P}^1}}					
\newcommand{\sinfty}{\ensuremath{^{S,\infty}}}
\newcommand{\G}{\ensuremath{\mathcal{G}}}
\newcommand{\Ah}{\ensuremath{\mathcal{A}}}

\DeclareMathOperator{\Div}{Div}

\newcommand{\p}{\ensuremath{\mathfrak{p}}}
\newcommand{\q}{\ensuremath{\mathfrak{q}}}

\newcommand{\LI}{\mathcal{L}}
\newcommand{\No}{\ensuremath{N}}
\newcommand{\Tr}{\ensuremath{Tr}}

\newcommand{\too}{\longrightarrow}								
\newcommand{\mapstoo}{\longmapsto}

\newtheorem{Lem}{Lemma}[section]
\makeatletter
\newlength{\@thlabel@width}%
\newcommand{\thmenumhspace}{\settowidth{\@thlabel@width}{\itshape1.}\sbox{\@labels}{\unhbox\@labels\hspace{\dimexpr-\leftmargin+\labelsep+\@thlabel@width-\itemindent}}}
\makeatother
\newtheorem{Pro}[Lem]{Proposition}
\newtheorem{Thm}{Theorem}

\newtheorem{Def}[Lem]{Definition}
\newtheorem{Rem}[Lem]{Remark}
\newtheorem{Con}[Lem]{Conjecture}
\newtheorem{Cor}[Lem]{Corollary}

\newtheorem*{Hyp}{Hypothesis}

\author[L. Gehrmann]{Lennart Gehrmann}
\address{L. Gehrmann \\ Fakult\"at f\"ur Mathematik \\ Universit\"at Duisburg-Essen \\ Thea-Leymann-Stra\ss e 9 \\ 45127 Essen \\ Germany}
\email{lennart.gehrmann@uni-due.de}

\title[Functoriality of automorphic L-invariants]{Functoriality of automorphic L-invariants and applications}

\begin{document}

\begin{abstract}
We study the behaviour of automorphic L-invariants associated to cuspidal representations of GL(2) of cohomological weight 0 under abelian base change and Jacquet-Langlands lifts to totally definite quaternion algebras.
Under a standard non-vanishing hypothesis on automorphic L-functions and some technical restrictions on the automorphic representation and the base field we get a simple proof of the equality of automorphic and arithmetic L-invariants.
This together with Spie\ss' results on p-adic L-functions yields a new proof of the exceptional zero conjecture for modular elliptic curves - at least, up to sign.
\end{abstract}

\maketitle

\tableofcontents
\input{BS-Introduction.tex}

\section{Preliminaries}
Throughout this section we fix a finite extension $F$ of $\Q_p$ with valuation ring $\OO$.
We write $\ord\colon F^{\ast}\to \Z$ for the normalized valuation, i.e.~$\ord(\varpi)=1$ holds for any uniformizer $\varpi$ of $F$.
\input{BS-P-LI.tex}

\input{BS-P-Steinberg.tex}

\section{Automorphic L-invariants}
For the rest of the article we fix a number field $F$ with ring of integers $\OO_F$.
If $v$ is a place of $F$, we denote by $F_{v}$ the completion of $F$ at $v$.
If $\p$ is a finite place, we let $\OO_{F_\p}$ denote the valuation ring of $F_\p$ and write $\ord_{\p}$ for the normalized valuation.
For a rational place $l\leq \infty$ we denote by $S_{l}=S_l(F)$ the set of places of $F$ lying above $l$.

Given a finite (possibly empty) set $S$ of places of $F$ we define the ``$S$-truncated adeles'' $\A_F^{S}$ as the restricted product of the completions $F_{v}$ over all places $v$ which are not in $S$. We often write $\A_F\sinfty$ instead of $\A_F^{S\cup S_{\infty}}$.
If $G$ is an algebraic group over $F$ and $v$ is a place of $F$, we write $G_v=G(F_v)$ and put $G_S=\prod_{v\in S} G_v$. 
If $K\subseteq G(\A)$ is a subgroup, we define $K^{S}$ as the image of $K$ under the quotient map $G(\A_F)\to G(\A_F^{S})$.

Further, we fix an inner form $\GG$ of the algebraic group $\GL_2/F$.
We will denote the centre of $\GG$ by $Z$ and put $\PGG=\GG/Z$.
If $\GG$ is split, we always identify $\GG$ and $\GL_2$.
Similarly, if $v$ is a place of $F$ at which $\GG$ is split we choose an isomorphism of $\GG_v$ with $\GL_2(F_v)$.
Given a place $l\leq\infty$ of $\Q$ we write $S_{l}(\GG)\subseteq S_l$ for the set of places of $F$ above $l$ at which $\GG$ is split.
We put
\begin{align*}
d_{\GG}=\begin{cases}
|S_{\infty}(\GG)|-1 & \mbox{if}\ \GG\ \mbox{is split,}\\
|S_{\infty}(\GG)| & \mbox{else.}
\end{cases}
\end{align*}

Following Spie\ss~(cf.~\cite{Sp}) we define automorphic $\LI$-invariants attached to cuspidal automorphic representations of $\GG(\A_F)$ which are cohomological with respect to the trivial coefficient system.
For split $\GG$ all omitted proofs can essentially be found in Section 4 and 5 of \textit{loc.~cit.~}
See also \cite{De}, Section 4, for the case of a non-trivial central character and \cite{BG2} respectively \cite{Felix} for the case that the group $\GG$ is non-split. 

\input{BS-L-MS.tex}
\input{BS-L-Definition.tex}

\section{Exercises in base change}
From now on we fix a cuspidal automorphic representation $\pi$ of $\GL_2(\A_F)$, which is cohomological with respect to the trivial coefficient system.
We want to study the behaviour of automorphic $\LI$-invariants attached to $\pi$ under base change.
To this end let $E$ be a finite abelian extension of $F$ with Galois group $\G$.
We will always identify characters of $\G$ with idele class characters via the Artin reciprocity map.

Let $\pi_E$ the base change of $\pi$ to $\GL_2(\A_E)$.
By the cyclic base change theorem of Arthur and Clozel (cf.~\cite{AC}) we know that $\pi_E$ is an automorphic representation.
If $\pi$ is Steinberg at a prime $\p$ of $F$, we have that $\pi_E$ is Steinberg at all primes of $E$ that divide $\p$ (see \cite{AC},  Lemma 6.12).
Therefore, $\pi_E$ is cuspidal if $\pi_\p$ is Steinberg for at least one finite place $\p$ of $F$.
Thus, we may always assume that $\pi_E$ is cuspidal.
\input{BS-BS-Galois.tex}
\input{BS-BS-Artin.tex}

\input{BS-BS-Anti.tex}

\input{BS-BS-C.tex}

\bibliographystyle{abbrv}
\bibliography{bibfile}

\end{document}

%% file: BS-Introduction.tex
\section*{Introduction}
In his seminal article \cite{D} Darmon defined automorphic $\LI$-Invariants of modular forms of weight $2$.
His definition was inspired by Teitelbaum's work (cf.~\cite{Teitelbaum}) on $\LI$-invariants for automorphic forms on definite quaternion groups.
The construction was generalized to various settings, for example:
to modular forms of higher weight by Orton (cf.~\cite{Orton}), to Hilbert modular forms of parallel weight $2$ by Spie\ss~(cf.~\cite{Sp}) and to automorphic forms of parallel weight $2$ on quaternion groups over fields with narrow class number $1$ by Guitart, Masdeu and Şengün (cf.~\cite{GMS}).

An advantage of these automorphic $\LI$-invariants over their arithmetic counterparts is that they can directly be related to derivatives of $p$-adic $L$-functions.
However, this raises the question of whether one can show that automorphic and arithmetic $\LI$-invariants agree without using that both types of invariants show up in exceptional zero formulae. 
Furthermore, it is natural to ask whether automorphic $\LI$-invariants behave well under various automorphic operations like base change and Jacquet-Langlands lifts.
In the present note we want to give partial answers to these questions.

Let $\pi$ be a cuspidal automorphic representation of $\GL_2$ of parallel weight $2$ over a number field $F$.
Under a standard simultaneous non-vanishing hypothesis on $L$-functions and some technical assumptions we prove that $\LI$-invariants of $\pi$
\begin{itemize}
\item behave as expected under abelian base change (see Corollary \ref{bcc}),
\item are independent of the sign at infinity (see Theorem \ref{invariance}) and
\item agree up to sign with $\LI$-invariants of Jacquet-Langlands lifts of $\pi$ to totally definite quaternion groups (see Theorem \ref{haupt}).
\end{itemize}

Via the Cerednik-Drinfeld uniformization of Shimura curves one can relate $\LI$-invariants of automorphic representations on totally definite quaternion groups directly to arithmetic $\LI$-invariants (cf.~\cite{BG2}, Theorem 4.10).
Thus, combining our results with Spie\ss' work on $p$-adic $L$-functions gives a new proof of the exceptional zero formula for modular elliptic curves over totally real fields in certain situations (see Corollary \ref{ezc}).
Note that even in the case $F=\Q$ the proof is different from previous ones, as it relies neither on Hida families nor on Kato's Euler system. 

The article is structured as follows:
The first section is a collection of preliminaries on $p$-adic fields.
We introduce $\LI$-invariants of $p$-adic periods and recall Breuil's construction (cf.~\cite{Br}) of extensions of the Steinberg representation.

In the second section we define automorphic $\LI$-invariants and state known results.
Here, we follow Spie\ss' approach rather closely.

In Section \ref{Galois} we show that, if $E/F$ is an abelian extension, its Galois group acts on the set of automorphic $\LI$-invariants of the base change $\pi_E$ of $\pi$ to $E$ as expected.
This together with an Artin formalism for $p$-adic $L$-functions is used in Section \ref{Artin} to prove the invariance of $\LI$-invariants under base change.
If the base field is totally imaginary, the $\LI$-invariant does not depend on a choice of sign at infinity by definition.
For an arbitrary number field the independence of the sign at infinity follows by considering a base change to a suitable totally imaginary extension.

In case $F$ is totally real and $E/F$ is an imaginary quadratic extension, in which the primes under consideration are split, there are two different constructions of anticyclotomic $p$-adic $L$-functions.
One can either restrict the $p$-adic $L$-function of the base change $\pi_E$, which was constructed by Deppe in \cite{De}, to the anticyclotomic direction or one can use quaternionic Stickelberger elements associated to the Jacquet-Langlands lift $\pi_B$ of $\pi$ to a quaternion algebra $B$, in which $E$ can be embedded.
Comparing leading terms formula for quaternionic Stickelberger elements (see \cite{BG2}, Theorem 4.5) with Bergunde's generalization of Spie\ss' exceptional zero formula for the $p$-adic $L$-function of $\pi_E$ (see \cite{Felix}, Theorem 4.17) we relate $\LI$-invariants of the base change $\pi_E$ to $\LI$-invariants of the Jacquet-Langlands lift $\pi_B$ - at least up to sign (see Section \ref{Anti}).

Finally, in order to prove the above mentioned Theorem \ref{haupt} we use various non-vanishing results to twist into a situation, where the previous established results apply.
This in turn implies the equality of automorphic and arithmetic $\LI$-invariants for modular elliptic curves over totally real fields up to sign.
Moreover, we utilize the transcendence of the Tate periods of elliptic curves to remove the ambiguity of sign once one base changes to an imaginary quadratic extension as above (see Theorem \ref{CM}).
Note that the equality of automorphic $\LI$-invariants under Jacquet-Langlands lifts was proven before in several instances (see for example \cite{BDI} or \cite{DG}).
But our methods differ vastly from previous ones as we do not make any use of Hida families.

\bigskip
\textbf{Notations.} Throughout the article we fix a prime $p$.
All rings are assumed to be commutative and unital.
The group of invertible elements of a ring $R$ will be denoted by $R^{\ast}$.

If $R$ is a ring and $G$ is a group, we write $R[G]$ for the group algebra of $G$ with coefficients in $R$.
Given a group $G$ and a group homomorphism $\epsilon\colon G\to R^{\ast}$ we let $R(\epsilon)$ be the $R[G]$-module which underlying $R$-module is $R$ itself and on which $G$ acts via the character $\epsilon$.
If $M$ is another $R[G]$-module, we put $M(\epsilon)=M\otimes_{R}R(\epsilon)$.

If $X$ and $Y$ are topological spaces, we write $C(X,Y)$ for the space of continuous maps from $X$ to $Y$.
The cardinality of a finite set $X$ will be denoted by $|X|$.

If $E/F$ is a finite extension of fields, we let $\No_{E/F}\colon E^{\ast}\to F^{\ast}$ be the relative norm and $\Tr_{E/F}\colon E\to F$ the relative trace map.

\bigskip
\textbf{Acknowledgements.} It is my pleasure to thank Felix Bergunde for several helpful discussions.
I thank Santiago Molina for comments on an earlier draft of the article.

%% file: BS-P-LI.tex
\subsection{$p$-adic periods and $\LI$-invariants} \label{LI}
We fix a ring of integers $R$ of an algebraic number field.
In the following, we write the group law of the $R$-module $F^{\ast}\otimes R$ multiplicatively.
\begin{Def}
\begin{enumerate}[(a)]\thmenumhspace
\item An element of $F^{\ast}\otimes R$ is called a (p-adic) period if it is not in the kernel of the map
$$\ord\otimes \id_R\colon F^{\ast}\otimes R\too R.$$
\item Let $q\in F^{\ast}\otimes R$ be a period.
Given a finite extension $\Omega$ of $\Q_p$ and a continuous homomorphism $\lambda\colon F^{\ast} \to \Omega$ we define the $\LI$-invariant of $q$ with respect to $\lambda$ by
\begin{align*}
\LI_\lambda(q):=(\ord\otimes \id_R)(q)^{-1} \cdot (\lambda\otimes\id_R)(q)\in \Omega\otimes R.
\end{align*}
\end{enumerate}
\end{Def}
By definition we have $\LI_\lambda(q^{n})=\LI_\lambda(q)$ for every $n\in R-\left\{0\right\}$.
If $x\in F^{\ast}- \OO^{\ast}$, we write $\log_x\colon F^{\ast}\to F$ for the branch of the $p$-adic logarithm such that $\log_x(x)=0$ and set $\LI_x(q)=\LI_{\log_x}(q)$.

The proof of the following lemma is a straightforward computation.
\begin{Lem}\label{changeofdir}
Let $q\in F^{\ast}\otimes R$ be a period.
\begin{enumerate}[(a)]
\item\label{first} For elements $x,y\in F^{\ast}- \OO^{\ast}$ the formula
\begin{align*}
\LI_{y}(q)=\LI_{x}(q)-\LI_{x}(y)
\end{align*}
holds.
\item\label{change2} Let $\tilde{q}\in F^{\ast}\otimes R$ be another period. Then the following conditions are equivalent:
\begin{enumerate}[(i)]
\item There exist $n,m\in R-\left\{0\right\}$ such that $q^{n}=\tilde{q}^{m}$
\item $\LI_{\lambda}(q)=\LI_{\lambda}(\tilde{q})$ for all continuous homomorphism $\lambda\colon F^{\ast}\to\Omega$ and all finite extensions $\Omega$ of $\Q_p$
\item There exists an element $x\in F^{\ast}- \OO^{\ast}$ such that $\LI_x(q)=\LI_x(\tilde{q})$
\end{enumerate} 
\end{enumerate}
\end{Lem}

The following simple but cumbersome lemma will be needed later to ensure that the number of certain badly behaved cases is finite (see Theorem \ref{haupt}).
It should be skipped while reading the article for the first time.
\begin{Lem}\label{ugly}
Let $q^{B}\in F^{\ast}\otimes R$ be a period and $K$ an intermediate extension of $F/\Q_p$.
There exists a quadratic polynomial $f\in K[T]$ such that for all periods $q,q^{\prime}\in F^{\ast}\otimes R$ and elements $x,x^{\prime}\in F^{\ast}-\OO^{\ast}$ with
\begin{itemize}
\item $\LI_x(q)=-\LI_x(q^{B})$,
\item $\LI_{x^{\prime}}(q^{\prime})=-\LI_{x^{\prime}}(q^{B})$ and
\item $\LI_p(\No_{F/K}(q))^{2}= \LI_p(\No_{F/K}(q^{\prime}))^{2}$
\end{itemize}
we have:
$$f(\LI_p(\No_{F/K}(x)))=f(\LI_p(\No_{F/K}(x^{\prime}))).$$
\end{Lem}
\begin{proof}
By Lemma \ref{changeofdir} \eqref{first} we get $\LI_p(q)=-\LI_p(q^{B})+2\LI_p(x)$.
Since $p\in \Q_p\subseteq K$ this implies
\begin{align*}
\LI_p(\No_{F/K}(q))&=\Tr_{F/K}(\LI_p(q))\\
&=\Tr_{F/K}(-\LI_p(q^{B})+2\LI_p(x))\\
&=-\LI_p(\No_{F/K}(q^{B}))+2\LI_p(\No_{F/K}(x)).
\end{align*}
We can do the same calculation for $q^{\prime}$ and $x^{\prime}$.
Therefore, the equation in the third condition yields
\begin{align*}
&\LI_p(\No_{F/K}(x))^{2}\hspace{0.3em}-\LI_p(\No_{F/K}(q^{B}))\LI_p(\No_{F/K}(x))\\
=&\LI_p(\No_{F/K}(x^{\prime}))^{2}-\LI_p(\No_{F/K}(q^{B}))\LI_p(\No_{F/K}(x^{\prime})),
\end{align*}
which is the quadratic identity we were after.
\end{proof}



%% file: BS-P-Steinberg.tex
\subsection{Extensions of the Steinberg representation} \label{Steinberg}
Given a prodiscrete group $A$ we define the $A$-valued (continuous) Steinberg representation $\St(A)$ of $\PGL_2(F)$ as the space of continuous $A$-valued functions on $\PP(F)$ modulo constant functions.

For a continuous group homomorphism $\lambda\colon F^\ast\to A$ we define $\widetilde{\mathcal{E}}(\lambda)$ as the set of pairs $(\Phi,r)\in C(\GL_2(F),A)\times \Z$ with
$$
\Phi\left(g\cdot\begin{pmatrix}
t_1 & u \\ 0 & t_2
\end{pmatrix}\right)
=\Phi(g) + r\cdot \lambda(t_1)
$$
for all $t_1,t_2\in F^\ast$, $u\in F$ and $g\in\GL_2(F)$.
The group $\GL_2(F)$ acts on $\widetilde{\mathcal{E}}(\lambda)$ via $g.(\Phi(-),r)=(\Phi(g^{-1}\cdot-),r)$.
The subspace $\widetilde{\mathcal{E}}(\lambda)_0$ of tuples of the form $(\Phi,0)$ with constant $\Phi$ is $\GL_2(F)$-invariant.
We get an induced $\PGL_2(F)$-action on the quotient $\mathcal{E}(\lambda)=\widetilde{\mathcal{E}}(\lambda) /\widetilde{\mathcal{E}}(\lambda)_0$.

\begin{Lem}
Let $\pi\colon \PGL_{2}(F)\to \PP(F),\ g\mapsto g.\infty$, be the canonical projection. The following sequence of $\Z[\PGL_{2}(F)]$-modules is exact:
\begin{align*}
0\too\St(A)\xlongrightarrow{(\pi^{\ast},0)}\mathcal{E}(\lambda)\xlongrightarrow{(0,\id_\Z)}\Z\too 0
\end{align*}
We write $b_{\lambda}$ for the associated cohomology class in $\HH^1(\PGL_{2}(F),\St(A))$.
\end{Lem}
\begin{proof}
See Lemma 3.11 (a) of \cite{Sp}.
Note that we use a slightly different extension than the one in \textit{loc.~cit.~}
\end{proof}

%% file: BS-L-MS.tex
\subsection{Modular symbols} \label{Modular Symbols}
We write $\Div(\PP(F))$ for the free abelian group on $\PP(F)$ and $\Div_0(\PP(F))$ for the kernel of the map $$\Div(\PP(F)) \to \Z,\ \sum_P m_P P \mapsto \sum_P m_P.$$ The $\GL_2(F)$-action on $\PP(F)$ induces an action on $\Div_0(\PP(F))$.
We put
$$
D_{\GG}=\begin{cases}
\Div_0(\PP(F)) & \mbox{if}\ \GG\ \mbox{is split},\\
\Z & \mbox{else.}
\end{cases}
$$
Given a place $\p\in S_p(\GG)$ and a prodiscrete abelian group $A$ we write $\St_\p(A)$ for the $A$-valued continuous Steinberg representation of $\GG_\p$.
For a subset $S\subseteq S_p(\GG)$  we put $\St_S(A)=\bigotimes_{\p \in S} \St_\p(A)$. Moreover, if $A=\Z$ with the discrete topology we simply write $\St_S=\St_S(\Z)$.
Given the data
\begin{itemize}
\item a ring $R$ and an $R$-module $N$ (equipped with the trivial $G(F)$-action),
\item subsets $S_0\subseteq S\subseteq S_p(\GG)$,
\item an open compact subgroup $K\subseteq \GG(\A_{F}^{S,\infty})$ and
\item a locally constant idele class character $\omega\colon \I_F/F^\ast\to R^{\ast}$ such that $\omega_v$ is trivial for all $v\in S\cup S_\infty$
\end{itemize} we define $\Ah(K,\omega,S_0;N)^{S}$ to be the space of all functions
\begin{align*}
\Phi\colon \GG(\A_F\sinfty)\too\Hom_\Z(\St_{S_0},\Hom_\Z(D_{\GG},N))
\end{align*}
such that $\Phi(gzk)=\omega(z)\Phi(g)$ for all $z\in Z(\A_F\sinfty)$, $k\in K$ and $g\in \GL_2(\A_F\sinfty)$.
By our assumptions the natural $\GG(F)$-action on $\Ah(K,\omega,S_0;N)^{S}$ factors through $\PGG(F)$.

\begin{Def}
For a datum as above and a continuous character $\epsilon\colon \PGG_\infty\to \{\pm 1\}$ we let $$\MS_{\GG}^{i}(K,\omega,S_0;N)^{S,\epsilon} = \HH^{i}(\PGG(F),\Ah(K,\omega,S_0,N)^{S}(\epsilon))$$ be the space of $N$-valued $S$-augmented $S_0$-special modular symbols of level $K$, central character $\omega$ and sign $\epsilon$.
Furthermore, we consider
$$\MS_{\GG}^{i}(\omega,S_0;N)^{S,\epsilon}= \varinjlim_K \MS_{\GG}^{i}(K,\omega,S_0;N)^{S,\epsilon}$$
with its natural $\GG(\A_F\sinfty)$-action.
\end{Def}

\begin{Pro}\label{FlachundNoethersch}
Given a datum as above we have:
\begin{enumerate}[(a)]
\item The $R$-module $\MS_{\GG}^i(K,\omega,S_0;R)^{S,\epsilon}$ is finitely generated for all $i\geq 0$ if $R$ is Noetherian.
\item\label{FN2} If $N$ is a flat $R$-module, then the canonical map
	\begin{align*}
	 \MS_{\GG}^i(K,\omega,S_0;R)^{S,\epsilon} \otimes_R N \to \MS_{\GG}^i(K,\omega,S_0;N)^{S,\epsilon}
	\end{align*}
	is an isomorphism for all $i \geq 0$.
\end{enumerate}
\end{Pro}
\begin{proof}
If $\GG=\PGL_2$ and the field $F$ is totally real, this is Proposition 4.6 of \cite{Sp}. The proof also works in our more general setup.
\end{proof}

%% file: BS-L-Definition.tex
\subsection{Definition and basic properties} \label{Definition}
Let $\pi=\otimes_v\pi_v$ be a cuspidal representation of $\GG(\A_{F})$, which is cohomological with respect to the trivial coefficient system. 
By a result of Clozel (cf.~\cite{C}) there exists a smallest subfield $\Q_\pi\subseteq \C$ over which $\pi^{\infty}$ can be defined and $\Q_\pi$ is a finite extension of $\Q$.
Let $\mathcal{R}_\pi$ be the ring of integers of $\Q_\pi$ and $V_\pi$ the model of $\pi^\infty$ over $\Q_{\pi}$.
The Archimedean components of the central character $\omega_\pi$ of $\pi$ are trivial.
Therefore, $\omega_{\pi}$ takes values in $\Q_\pi^{\ast}$.
If $S\subseteq S_p(\GG)$ is a subset, $\Omega$ a $\Q_{\pi}$-algebra and $M$ an $\Omega[\GL_2(\A_F\sinfty)]$-module we put
\begin{align*}
M_{\pi}=\Hom_{\Omega[\GG(\A_F\sinfty)]}(V_\pi^{S}\otimes\Omega,M).
\end{align*}
Let $S_p(\pi)\subseteq S_p(\GG)$ denote the subset of those primes $\p\in S_p(\GG)$ such that $\pi_\p$ is the Steinberg representation.
As a straightforward generalization of \cite{Sp}, Proposition 4.8, we get:
\begin{Pro}
Let $\Omega$ be a $\Q_{\pi}$-algebra, $\epsilon\colon \PGG_\infty\to \left\{\pm 1\right\}$ a character and $S$ a subset of $S_p(\pi)$. Then, we have
\begin{align*}
\MS_{\GG}^i(\omega_\pi,S;\Omega)^{S,\epsilon}_{\pi}\cong
\begin{cases}
\Omega & \mbox{if}\ i=d_{\GG}, \\
0 & \mbox{if}\ i\leq d_{\GG}-1.
\end{cases}
\end{align*}
\end{Pro}

For the remainder of this section we fix a place $\tilde{\p} \in S_p(\pi)$. 
Let $\Omega$ be a finite extension of $\Q_p$.
Every element in $\Hom(\St_{\tilde{\p}},\mathcal{R}_\pi)$ can be uniquely extended to a continuous functional on $\St_{\tilde{\p}}(\Omega\otimes \mathcal{R}_\pi)$. Thus we get a pairing
\begin{align*}
\Ah(K,\omega_\pi,S;\mathcal{R}_\pi)^{S}\times \St_{\tilde{\p}}(\Omega) \too 
\Ah(K,\omega_\pi,S-\left\{\tilde{\p}\right\};\mathcal{R}_\pi)^{S}\otimes \Omega
\end{align*}
for every open compact subgroup $K\subseteq \GG(\A_F\sinfty)$ and every subset $S\subseteq S_p(\pi)$ which contains $\tilde{\p}$. Hence, using Proposition \ref{FlachundNoethersch} \eqref{FN2} we get a cup product pairing
\begin{align*}
\MS_{\GG}^{i}(\omega_\pi,S;\Omega\otimes\mathcal{R}_\pi)^{S,\epsilon}\times \HH^{j}(\PGL_2(F),\St_{\tilde{\p}}(\Omega))
\xlongrightarrow{\cup}\MS_{\GG}^{i+j}(\omega_\pi,S-\left\{\tilde{\p}\right\};\Omega\otimes \mathcal{R}_\pi)^{S,\epsilon}
\end{align*}
which commutes with the $\GG(\A_F\sinfty)$-action.

Let $\lambda\colon F_{\tilde{\p}}^{\ast}\to\Omega$ be a continuous homomorphism and $b_{\lambda}\in \HH^{1}(\PGL_2(F),\St_{\tilde{\p}}(\Omega))$ be the (restriction of the) cohomology class associated to $\lambda$ in Section \ref{Steinberg}. 
Taking the cup product with $b_{\lambda}$ induces a map
$$
b_{\lambda}(\pi)^{\epsilon}\colon \MS_{\GG}^{d_{\GG}}(\omega_\pi,S;\Omega\otimes\mathcal{R}_\pi)^{S,\epsilon}_{\pi}\xrightarrow{\_\cup b_{\lambda}} \MS_{\GG}^{d_{\GG}+1}(\omega_\pi,S-\left\{\tilde{\p}\right\};\Omega\otimes \mathcal{R}_\pi)^{S,\epsilon}_{\pi}.
$$

Similarly as in Lemma 5.2 of \cite{Sp} one can prove the following lemma.
\begin{Lem}
Let $\ord_{\tilde{\p}}\colon F_{\tilde{\p}}^{\ast}\to \Z\subseteq \Omega$ be the normalized valuation.
Then $b_{\ord_{\tilde{\p}}}(\pi)^{\epsilon}$ is an isomorphism of free $\Omega\otimes\mathcal{R}_\pi$-modules of rank one.
\end{Lem}

\begin{Def}
The (automorphic) $\LI$-invariant $\LI_{\lambda}(\pi,\tilde{\p})^{\epsilon}\in\Omega\otimes\mathcal{R}_\pi$ of $\pi$ at $\p$ with respect to $\lambda$ and sign $\epsilon$ is defined by the relation
$$b_{\lambda}(\pi)^{\epsilon}=\LI_{\lambda}(\pi,\tilde{\p})^{\epsilon}\cdot b_{\ord_{\tilde{\p}}}(\pi)^{\epsilon}.$$
In case $\lambda=\log_p \circ \No_{F_{\tilde{\p}}/\Q_p}$ we put $\LI^{\cyc}(\pi,\tilde{\p})^{\epsilon}=\LI_{\lambda}(\pi,\tilde{\p})^{\epsilon}.$
In case $\lambda=\log_x$ with $x\in F_{\tilde{\p}}^{\ast}-\mathcal{O}_{\tilde{\p}}^{\ast}$ we put $\LI_x(\pi,\tilde{\p})^{\epsilon}=\LI_{\lambda}(\pi,\tilde{\p})^{\epsilon}.$
If $\epsilon$ is the trivial character, we drop it from the notation.
\end{Def}
It is easy to see that $\LI_{\lambda}(\pi,\tilde{\p})^{\epsilon}$ is independent of the choice of set $S\subseteq S_p(\pi)$ containing $\tilde{\p}$.
Note that in case that $F$ is totally imaginary, automorphic $\LI$-invariants do not depend on a choice of sign at infinity.

\begin{Pro}[Invariance under twists]\label{twist} Let $\chi\colon \I_F/F^{\ast}\to \C^{\ast}$ be a locally constant character with $\chi_{\tilde{\p}}=1$. Then we have the following equality:
\begin{align*}
\LI_{\lambda}(\pi,\tilde{\p})^{\epsilon}=\LI_{\lambda}(\pi\otimes\chi,\tilde{\p})^{\epsilon\chi_\infty}
\end{align*}
\end{Pro}
\begin{proof}
A special case of the assertion is proven in \cite{Sp}, Lemma 5.5.
The same proof works in the more general setting.
\end{proof}

\begin{Lem}[Automorphic periods]\label{autperiod} There exists a period $q(\pi,\tilde{\p})^{\epsilon}\in F_{\tilde{\p}}^{\ast}\otimes \mathcal{R}_\pi$
such that
$$\LI_{\lambda}(\pi,\tilde{\p})^{\epsilon}=\LI_{\lambda}(q(\pi,\tilde{\p})^{\epsilon})$$
holds for all continuous homomorphism $\lambda\colon F_{\tilde{\p}}^{\ast}\to\Omega$ and all finite extensions $\Omega$ of $\Q_p$.
\end{Lem}
\begin{proof}
In case that the central character is trivial the automorphic period is constructed in the proof of\cite{BG2}, Theorem 4.5.
The construction can be generalized verbatim to our situation.
\end{proof}

%% file: BS-BS-Galois.tex
\subsection{Galois action} \label{Galois}
The Galois group $\G$ acts on $\GL_2(\A_E^{\infty})$ via its action on the coefficients.
If $\q$ is a prime of $E$ and $\sigma$ is an element of $\G$, we set $\q^{\sigma}=\sigma^{-1}(\q)$.
If $\lambda\colon E_\q^{\ast}\to \Omega$ is a continuous homomorphism with values in a finite extension of $\Q_p$, we put
$$\lambda^{\sigma}\colon E_{\q^{\sigma}}^{\ast}\xlongrightarrow{\sigma}E_\q^{\ast}\xlongrightarrow{\lambda}\Omega.$$
For every locally discrete group $A$ the isomorphism $\PP(E_{\q^{\sigma}})\xrightarrow{\sigma}\PP(E_\q)$ induces an isomorphism
$$\sigma_\ast\colon \St_{\q^{\sigma}}(A)\too\St_{\q}(A).$$
Similarly, we get an isomorphism
$$\sigma_\ast\colon \Div_0(\PP(E))\too \Div_0(\PP(E)),\ \sum_P m_P P\mapstoo \sum_P m_P \sigma(P).$$
Given a finite set $S=\left\{\q_1,\ldots,\q_r\right\}$ of finite place we put $S^{\sigma}=\left\{\q_1^\sigma,\ldots,\q_r^\sigma\right\}$.
If $\epsilon\colon\PGL_2(E_\infty)\to \left\{\pm 1\right\}$ is a locally constant character, we set
$$\epsilon^{\sigma}\colon \PGL_2(E_\infty)\xlongrightarrow{\sigma}\PGL_2(E_\infty)\xlongrightarrow{\epsilon}\left\{\pm 1\right\}.$$

Let $S_p(\pi_{E})\subseteq S_p(E)$ denote the set of all primes $\q$ of $E$ which divide $p$ and such that $\pi_{E,\q}$ is Steinberg. 
Given subsets $S_0\subseteq S\subseteq S_p(\pi_{E})$, an open compact subgroup $K\subset \GL_2(\A_E^{S,\infty})$, a locally constant character $\epsilon\colon \PGL_2(E_\infty)\to \left\{\pm 1\right\}$ and an $\mathcal{R}_\pi$-module $N$ the map
$$\sigma\colon\Ah(K,\omega_{\pi_{E}},S_0;N)^{S}(\epsilon)
\too\Ah(\sigma^{-1}(K),\omega_{\pi_{E}},S_0^{\sigma};N)^{S^{\sigma}}(\epsilon^{\sigma})$$
given by
$$(\sigma.\Phi)(g)(f)(D)=\Phi(\sigma(g))(\sigma_\ast(f))(\sigma_\ast D)$$
for $g\in \GL_2(\A_E^{S^{\sigma},\infty})$, $f\in \St_{S_0^{\sigma}}$ and $D\in \Div_0(\PP(E))$,
is an isomorphism, which is compatible with the homomorphism
$$
\PGL_2(E)\too \PGL_{2}(E),\ g\mapstoo \sigma(g).
$$
Therefore, we get a map in cohomology
\begin{align*}
\MS_{\GL_{2,E}}^{i}(\omega_{\pi_{E}},S_0;\mathcal{R}_\pi\otimes\Omega)^{S,\epsilon}
\xlongrightarrow{\sigma}\MS_{\GL_{2,E}}^{i}(\omega_{\pi_{E}},\sigma^{-1}(S_0);\mathcal{R}_\pi\otimes\Omega)^{S^{\sigma},\epsilon^{\sigma}},
\end{align*}
which is compatible with the isomorphism
$$\GL_2(\A_E\sinfty)\too \GL_2(\A_E^{S^{\sigma},\infty}),\ g\mapstoo \sigma^{-1}(g).$$

\begin{Lem}[Galois invariance]\label{gallemma}
Let $\epsilon\colon \PGL_2(E_\infty)\to \left\{\pm 1\right\}$ be a locally constant character, $\q\in S(\pi_{E})$ a prime and $\lambda\colon E^{\ast}_\q\to \Omega$ a continuous character with values in a finite extension of $\Q_p$.
Then the equality
$$
\LI_{\lambda}(\pi_E,\q)^{\epsilon}=\LI_{\lambda^{\sigma}}(\pi_E,\q^{\sigma})^{\epsilon^{\sigma}}
$$
holds for all $\sigma\in\G$.
\end{Lem}
\begin{proof}
The diagram
\begin{center}
  \begin{tikzpicture}
    \path 	(0,0) 	node[name=A]{$\MS_{\GL_{2,E}}^{i}(\omega_{\pi_{E}},\left\{\q\right\};\mathcal{R}_\pi\otimes\Omega)^{\left\{\q\right\},\epsilon}_{\pi_E}$}
		(0,-0.5)   node{$\times$}
		(0,-1) 	node[name=B]{$\HH^{1}(\PGL_2(E),\St_\q(\Omega))$}
		(0,-2.2) 	node[name=C]{$\MS_{\GL_{2,E}}^{i+1}(\omega_{\pi_{E}},\emptyset;\mathcal{R}_\pi\otimes\Omega)^{\left\{\q\right\},\epsilon}_{\pi_E}$}
		(7,0) 	node[name=D]{$\MS_{\GL_{2,E}}^{i}(\omega_{\pi_{E}},\left\{\q^{\sigma}\right\};\mathcal{R}_\pi\otimes\Omega)^{\left\{\q^{\sigma}\right\},\epsilon^{\sigma}}_{\pi_E}$}
		(7,-0.5)		node{$\times$}
		(7,-1) 	node[name=E]{$\HH^{1}(\PGL_2(E),\St_{\q^{\sigma}}(\Omega))$}
		(7,-2.2) 	node[name=F]{$\MS_{\GL_{2,E}}^{i+1}(\omega_{\pi_{E}},\emptyset;\mathcal{R}_\pi\otimes\Omega)^{\left\{\q^{\sigma}\right\},\epsilon^{\sigma}}_{\pi_E}$};
    \draw[->] (B) -- (C) node[midway, left]{$\cup$};
    \draw[->] (E) -- (F) node[midway, left]{$\cup$};
    \draw[->] (A) -- (D) node[midway, above]{$\sigma$};
    \draw[->] (B) -- (E) node[midway, above]{$\sigma_\ast$};
		\draw[->] (C) -- (F) node[midway, above]{$\sigma$};
  \end{tikzpicture} 
 \end{center}
is commutative and thus, the claim follows from the Galois invariance of $\pi_E$.
\end{proof}

We can view a character of $\PGL_2(F_\infty)$ also as a character of $\PGL_2(E_\infty)$ by precomposing it with the norm.
\begin{Cor}\label{decent}
Let $\epsilon\colon \PGL_2(F_\infty)\to\left\{\pm 1\right\}$ be a locally constant character and $\q\in S_p(\pi_E)$ a prime. The automorphic period $q(\pi,\q)^{\epsilon}$ of Lemma \ref{autperiod} can be chosen to be invariant under the decomposition group $\G_\q\subseteq \G$ of $\q$.
\end{Cor}
\begin{proof}
For $\sigma\in \G_\q$ the previous lemma implies that
$$\LI_{\lambda}(q(\pi,\q)^{\epsilon})=\LI_{\lambda}(\sigma(q(\pi,\q)^{\epsilon}))$$
holds for every continuous character $\lambda\colon E^{\ast}_\q\to \Omega$ with values in a finite extension of $\Q_p$ and any choice of automorphic period $q(\pi,\q)^{\epsilon}$.
By Lemma \ref{changeofdir} \eqref{change2} there exist non-zero $m,n\in \mathcal{R}_\pi$ such that $(q(\pi,\q)^{\epsilon})^{n}=(\sigma(q(\pi,\q)^{\epsilon}))^{m}$.
Considering the $p$-adic valuations of both periods we see that $m=n$ and therefore, $(q(\pi,\q)^{\epsilon})^{m}$ is invariant under $\sigma$.
\end{proof}

%% file: BS-BS-Artin.tex
\subsection{Artin formalism - the cyclotomic case}\label{Artin}
The $L$-function of $\pi_E$ factors as a product of twists of the $L$-function of $\pi$:
\begin{align}\label{product}
L(\pi_E,s)=\prod_{\chi\colon\G\to\C^{\ast}} L(\pi,\chi,s)
\end{align}
Let us assume that $\pi$ is $p$-ordinary, i.e.~$\pi_\p$ is either an ordinary, unramified principal series representation or an unramified twist of the Steinberg representation for all $\p\in S_p$.
Then, Deppe constructs in \cite{De} a multi-variable $p$-adic L-function $L_p(\pi,\underline{s})$. It fulfils the interpolation property
\begin{align}\label{interpolation}
L_p(\pi,\chi)=\tau(\chi)\prod_{\p \in S_p}e(\pi_\p,\chi_\p)L(\pi,1/2)
\end{align}
for all idele class characters of finite order which are unramified outside all places above $p$ and $\infty$. Here $\tau(\chi)$ is the Gauss sum of $\tau$ and $e(\pi_\p,\chi_\p)$ is an explicit modified Euler factor.
The modified Euler factor $e(\pi_\p,\chi_\p)$ vanishes if and only if $\pi_\p\otimes\chi_\p$ is of the following type:
either it is isomorphic to the Steinberg representation or to a parabolic induction from a character of the form
$$(F_\p^\ast)^{2}\too \C^{\ast},\ (t_1,t_2)\mapstoo \chi_2(t_2).$$
We call these exceptional principal series.
Note that if the central character of $\pi$ is trivial, no exceptional principal series can occur.

We write $L_p^{\cyc}(\pi,s)$ for the restriction of the $p$-adic $L$-function to the cyclotomic $\Z_\p$-extension.
If $E/F$ is unramified at all places $\p\in S_p$, the representation $\pi_E$ is also $p$-ordinary and thus, we can consider the $p$-adic $L$-function associated to $\pi_E$.
The product formula \eqref{product} together with the interpolation property \eqref{interpolation} yields
\begin{align}\label{pproduct}
L_p^{\cyc}(\pi_E,s)=\prod_{\chi\colon\G\to\C^{\ast}} L_p^{\cyc}(\pi\otimes\chi,s).
\end{align}

\begin{Lem}\label{comparison}
Assume that
\begin{itemize}
\item $\pi$ is $p$-ordinary,
\item $E/F$ is unramified at all primes $\p\in S_p$,
\item $S_{p}(\pi)=\left\{\tilde{\p}\right\}$ and every $\q \in S_p(\pi_{E})$ lies above $\tilde{\p}$
\item no local component of $\pi_{E}$ at a prime above $p$ is an exceptional principal series and
\item $L(\pi_E,1/2)\neq 0.$
\end{itemize}
Then we have the following equality:
\begin{align*}
\prod_{\substack{\chi\colon\G\to\C^{\ast}\\ \chi_{\tilde{\p}}=1}} f_{\tilde{\p}}\cdot \LI^{\cyc}(\pi\otimes\chi,\tilde{\p})
= \prod_{\q\in S_p(\pi_E)}  \LI^{\cyc}(\pi_E,\q),
\end{align*}
where $f_{\tilde{\p}}$ is the inertia degree of $\tilde{\p}$ in $E/F$.
\end{Lem}
\begin{proof}
The vanishing of the modified Euler factor for the Steinberg representation implies that
$$\ord_{s=0}L_p^{\cyc}(\pi\otimes\chi,s)\geq 1$$
if $\chi_{\tilde{\p}}$ is trivial.
Hence, by applying the $r=|S_p(\pi_{E})|$-th derivative to \eqref{pproduct} we get
\begin{align*}
\frac{1}{r!}\left.\frac{d^{r}}{ds^{r}}L_p^{\cyc}(\pi_E,s)\right|_{s=0}=
\prod_{\substack{\chi\colon\G\to\C^{\ast}\\ \chi_{\tilde{\p}}=1}} \left.\frac{d}{ds}L_p^{\cyc}(\pi\otimes\chi,s)\right|_{s=0}
\cdot \prod_{\substack{\chi\colon\G\to\C^{\ast}\\ \chi_{\tilde{\p}}\neq 1}} L_p^{\cyc}(\pi\otimes\chi,0).
\end{align*}
By the main theorem of \cite{Sp}, or rather its generalization to arbitrary number fields by Bergunde (see \cite{Felix}, Theorem 4.17), we get
$$\left.\frac{d}{ds}L_p^{\cyc}(\pi\otimes\chi,s)\right|_{s=0}=\LI^{\cyc}(\pi\otimes\chi,\tilde{\p})\hspace{0.8em}\cdot \prod_{\p \in S_p-\left\{\tilde{\p}\right\}}\hspace{-0.8em} e(\pi_{\p},\chi_{\p}) L(\pi,\chi,1/2)$$
if $\chi_{\tilde{\p}}$ is trivial, and
$$\frac{1}{r!}\left.\frac{d^{r}}{ds^{r}}L_p^{\cyc}(\pi_E,s)\right|_{s=0}
=\prod_{\q\in S_p(\pi_E)}  \hspace{-0.8em}\LI^{\cyc}(\pi_E,\q) \hspace{0.7em} \cdot \hspace{-0.8em} \prod_{\q\in S_p(E)\backslash S_p(\pi_E)} \hspace{-2em}e(\pi_{E,\q},1) \cdot L(\pi_E,1/2).$$
Note, that in \cite{Sp} and \cite{Felix} the representation $\pi$ is assumed to have trivial central character.
But the proof carries over to our more general setup.

The above formulas together with \eqref{product}, \eqref{interpolation} and our non-vanishing assumption on the special $L$-value imply that
$$\prod_{\q\in S_p(\pi_E)}  \hspace{-0.8em}\LI^{\cyc}(\pi_E,\q) \hspace{0.2em}
=\prod_{\substack{\chi\colon\G\to\C^{\ast}\\ \chi_{\tilde{\p}}=1}}\hspace{-0.8em}\LI^{\cyc}(\pi\otimes\chi,\tilde{\p}) \hspace{0.7em}\cdot
\prod_{\substack{\chi\colon\G\to\C^{\ast}\\ \chi_{\tilde{\p}}\neq 1}} \hspace{-1em}e(\pi_{\tilde{\p}},\chi_{\tilde{\p}}).
$$
We conclude by using $e(\pi_{\tilde{\p}},\chi_{\tilde{\p}})=1-\chi_{\tilde{\p}}(\varpi_{\tilde{\p}})$ and the formula
$\prod_{i=1}^{n-1}1-\zeta_{n}^{i}=n$
for a primitive $n$-th root of unity $\zeta_n$.
\end{proof}

\begin{Rem}
By slightly extending Deppe's construction, i.e.~allowing ramified twists of ordinary representations at places in $S_p$, one may allow $E/F$ to be ramified at primes above $p$.
One can also relax the ordinarity condition by applying recent work of Barrera, Dimitrov and Jorza (see \cite{BDJ}).

If $F=\Q$ and $E$ is an imaginary quadratic field in which the prime $p$ is inert, Barrera and Williams have proven the above lemma for modular forms of arbitrary weight (see \cite{BaWi}, Proposition 10.02).
\end{Rem}

In order to be able to twist to a situation, in which we can use the above lemma, we need to assume the following simultaneous non-vanishing hypothesis. 
\begin{Hyp}[SNV]
Let $\pi^{\prime}$ be a cuspidal automorphic representation of $\GL_2(F)$, which is cohomological with respect to the trivial coefficient system.
For every pair of locally constant characters $\eta\colon \A_F^{\ast}/F^{\ast}\to \left\{\pm 1 \right\}$ and $\epsilon\colon F_{\infty}^{\ast}\to \left\{\pm 1\right\}$ there exists a locally constant character $\chi^{\epsilon}\colon \A_F^{\ast}/F^{\ast}\to \C^{\ast}$ such that
\begin{itemize}
\item $\chi^{\epsilon}_{\infty}=\epsilon$,
\item $\chi^{\epsilon}_\p$ is trivial for all $\p\in S_p$ and
\item $L(\pi^{\prime},\chi^{\epsilon},1/2)\cdot L(\pi^{\prime},\eta\chi^{\epsilon},1/2)\neq 0$.
\end{itemize} 
\end{Hyp}
\begin{Rem}
In the case $F=\Q$ Akbary proved in \cite{Akbary} that there always exist infinitely many Dirichlet characters $\widetilde{\chi}$ such that the third condition holds. Unfortunately, he can not put restrictions on the local behaviour of $\widetilde{\chi}$.
On the other hand by a classic result of Rohrlich (cf.~\cite{Rohrlich}) Hypothesis (SNV) holds if $\eta$ is the trivial character.
\end{Rem}

The following result partially answers conjecture 5.4 of \cite{Sp}.
\begin{Thm}[Invariance of sign]\label{invariance}
Assume that Hypothesis (SNV) holds and that $\pi$ is $p$-ordinary.
Then for all $\tilde{\p}\in S_p(\pi)$ the automorphic $\LI$-invariant $\LI^{\cyc}(\pi,\tilde{\p})^{\epsilon}$ does not depend on the character $\epsilon$. 
\end{Thm}
\begin{proof}
By twisting with an appropriate character, which is unramified at $p$, we may assume that $S_p(\pi)=\left\{\tilde{\p}\right\}$ and that no local component of $\pi$ at a prime above $p$ is an exceptional principal series.
Let $E^{\prime}$ be a totally imaginary quadratic extension of $F$, which is inert at $\tilde{\p}$ and split at all other primes in $S_p$.
Let $\eta$ be the associated quadratic idele class character.
We choose a character $\chi$ (respectively $\chi^{\epsilon}$) as in Hypothesis (SNV) with respect to the trivial character (respectively $\epsilon$).
We may view $\chi$ (respectively $\chi^{\epsilon}$) as an idele class character on $\A_{E^{\prime}}^{\ast}$ by composing it with the norm map to $\A_F^{\ast}$.
We get the chain of equalities
\begin{align*}
\LI^{\cyc}(\pi,\tilde{\p})
&\stackrel{\ref{twist}}{=} \LI^{\cyc}(\pi\otimes\chi,\tilde{\p})\\
&\stackrel{\ref{comparison}}{=}\frac{1}{2}\cdot\LI^{\cyc}(\pi_{E^{\prime}}\otimes\chi,\tilde{\p})\\
&\stackrel{\ref{twist}}{=} \frac{1}{2}\cdot\LI^{\cyc}(\pi_{E^{\prime}}\otimes\chi^{\epsilon},\tilde{\p})\\
&\stackrel{\ref{comparison}}{=}  \LI^{\cyc}(\pi\otimes\chi^{\epsilon},\tilde{\p})\\
&\stackrel{\ref{twist}}{=}  \LI^{\cyc}(\pi,\tilde{\p})^{\epsilon},
\end{align*}
which proves the claim.
\end{proof}

\begin{Rem}
In the case $F=\Q$ the independence of $\LI$-invariants of the sign at infinity was proven by Breuil (see Corollary 4.5.4 of \cite{Br2}) using completed cohomology and by Bertolini, Darmon and Iovita (see Theorem 6.8 of \cite{BDI}) using Hida families.
\end{Rem}

The above theorem together with Lemma \ref{gallemma} and Proposition \ref{twist} immediately gives the following strengthening of Lemma \ref{comparison}.
\begin{Cor}[Invariance under base change]\label{bcc}
Assume that Hypothesis (SNV) holds or that $F$ is totally imaginary.
Under the assumptions of Lemma \ref{comparison} we have
$$(f_{\tilde{\p}}\cdot \LI^{\cyc}(\pi,\tilde{\p}))^{g_{\tilde{\p}}}
= \LI^{\cyc}(\pi_E,\q)^{g_{\tilde{\p}}},$$
where $\q$ is a prime of $E$ lying above $\tilde{\p}$ and $g_{\tilde{\p}}$ is the number of such primes.
\end{Cor}

\begin{Rem}
In certain situations one can relax some of the assumptions in the corollary above.
For example, in case $E/F$ is quadratic it is enough to assume that Hypothesis (SNV) holds, that $\pi$ is $p$-ordinary and that $E/F$ is unramified at all primes in $S_p$. 
\end{Rem}

\begin{Con}
Let $\tilde{\p}\in S_p(\pi)$ be a prime.
Then
$$e_{\tilde{\p}}f_{\tilde{\p}}\cdot \LI^{\cyc}(\pi,\tilde{\p})
= \LI^{\cyc}(\pi_E,\q)$$
holds for all primes $\q$ of $E$ above $\tilde{\p}$, where $e_{\tilde{\p}}$ denotes the ramification index of $\tilde{\p}$ in $E$.
\end{Con}

%% file: BS-BS-Anti.tex
\subsection{Artin formalism - the anticyclotomic case}\label{Anti}
From now we assume that $F$ is totally real and that $E$ is a totally imaginary quadratic extension with Galois group generated by $\sigma$.
We fix a prime $\tilde{\p}\in S_p(\pi)$, which we assume to be split in $E$.
The places of $E$ above $\tilde{\p}$ will be denoted by $\q_1$ and $\q_2$.
We choose an element $u\in (\mathcal{O}_E[\q_1^{-1}])^{\ast}-\mathcal{O}_E^{\ast}$, where $\mathcal{O}_E[\q_1^{-1}]$ denotes the ring of $\q_1$-integers.
\begin{Lem}\label{anticomparison}
Assume that
\begin{itemize}
\item $L(\pi_E,1/2)\neq 0$ and
\item there exists a Jacquet-Langlands lift $\pi_B$ of $\pi$ to the unit group of a totally definite quaternion algebra $B/F$ that is split at $\tilde{\p}$ and in which $E$ can be embedded. 
\end{itemize}
Then we have the following equality:
\begin{align*}
\LI_{u/\sigma(u)}(\pi_B,\tilde{\p})^2
=\LI_{u/\sigma(u)}(\pi_E,\q_1)\cdot \LI_{\sigma(u)/u}(\pi_E,\q_2)
=  \LI_{u/\sigma(u)}(\pi_E,\q_1)^2
\end{align*}
\end{Lem}
\begin{proof}
The second equality follows directly from Lemma \ref{gallemma}.
The proof of the first equality is similar to that of Lemma \ref{comparison}.
We give only a quick sketch of it.
Let $G(\tilde{\p})$ be Galois group of the maximal anticyclotomic extension of $E$ which is unramified outside $\q_1$ and $\q_2$.
Up to torsion $G(\tilde{\p})$ is isomorphic to
$$(\mathcal{O}^{\ast}_{E,\q_1}\times \mathcal{O}^{\ast}_{E,\q_2})/\mathcal{O}^{\ast}_{F,\tilde{\p}}
=(\mathcal{O}^{\ast}_{F,\tilde{\p}}\times \mathcal{O}^{\ast}_{F,\tilde{\p}})/\mathcal{O}^{\ast}_{F,\tilde{\p}}
\cong \mathcal{O}^{\ast}_{F,\tilde{\p}}.$$
Hence, by taking the $p$-adic logarithm we get up to torsion an isomorphism between $G(\tilde{\p})$ and $\mathcal{O}_{F,\tilde{\p}}$.

Let $L_{\tilde{\p}}^{\anti}(\pi_E,s)$ be the $p$-adic $L$-function, which interpolates the special values $L(\pi_E\otimes\chi,1/2)$ for characters $\chi\colon G(\tilde{\p})\to \C^{\ast}$.
It can be constructed by taking a projective limit of Stickelberger elements associated to $\pi_E$ and the split extension $E\times E$ (see \cite{Felix}, Definition 4.2).
Note that this is not quite a restriction of Deppe's $p$-adic $L$-function since we do not modify the Euler factors at other primes above $p$.

Comparing interpolation properties we get the equality
$$L_{\tilde{\p}}^{\anti}(\pi_E,s)=L_{\tilde{\p}}^{\anti}(\pi_B,s)L_{\tilde{\p}}^{\anti}(\pi_B,-s),$$
where $L_{\tilde{\p}}^{\anti}(\pi_B,s)$ is the $p$-adic $L$-function constructed via Stickelberger elements associated to $\pi_B$ and the non-split extension $E/F$ (see \cite{BG2}, Definition 3.2, or \cite{Felix}, Definition 4.2).
Using the main result of \cite{Felix} about leading terms of Stickelberger elements for both sides of the equation we get the desired identity.
\end{proof}

\begin{Rem}
The results of \cite{BG2} and \cite{Felix} are valid for arbitrary quaternion algebras.
But, in general it is harder to give an explicit description of $G(\tilde{\p})$ if the quadratic extension is not totally imaginary.
Still, under suitable assumptions we expect similar results to the one above.
For the sake of simplicity we chose to work only with totally imaginary extensions.
\end{Rem}

%% file: BS-BS-C.tex
\section{Conclusions}
Until the end of the article we assume that $F$ is a totally real number field.
\begin{Thm}[Invariance under Jacquet-Langlands lifts]\label{haupt}
Assume that Hypothesis (SNV) holds.
Let $\tilde{\p}\in S_p(\pi)$ be a prime.
If
\begin{itemize}
\item $\pi$ is $p$-ordinary,
\item $[F_{\tilde{\p}}:\Q_p]=1$ or $p$ does not split in $F$ and
\item there exists a Jacquet-Langlands lift $\pi_B$ of $\pi$ to the unit group of a totally definite quaternion algebra $B/F$, which is split at $\tilde{\p}$,
\end{itemize}
we have
$$\LI^{\cyc}(\pi,\tilde{\p})=\pm\LI^{\cyc}(\pi_B,\tilde{\p}).$$
\end{Thm}
\begin{proof}
Let $E$ be a totally imaginary quadratic extension of $F$, which is split at $\tilde{\p}$.
We may use Hypothesis (SNV) to assume that $L(\pi_E,1/2)\neq 0$, $S_p(\pi)=\tilde{\p}$ and all primes in $S_p(\pi_E)$ lie above $\tilde{\p}$.
Let $\q$ be one of those primes.
By Corollary \ref{bcc} and our assumptions we have
\begin{align}\label{almost}
\LI_p(\No_{F_{\tilde{\p}}/\Q_p}(q(\pi_E,\q)))^{2}=\LI_p(\No_{F_{\tilde{\p}}/\Q_\p}(q(\pi,\tilde{\p})))^{2}.
\end{align} 
By Lemma \ref{anticomparison} we have $\LI_{u/\sigma(u)}(q(\pi_E,\q))=\pm \LI_{u/\sigma(u)}(q(\pi_B,\tilde{\p}))$, where $u\in E^{\ast}$ is the element defined at the beginning of section \ref{Anti} and $\sigma$ is the generator of the Galois group of $E/F$.

If we assume that the sign is negative for every such extension $E$, then by Lemma \ref{ugly} the $p$-adic numbers $\LI_p(\No_{F_{\tilde{\p}}/\Q_p}(u/\sigma(u)))$ would all fulfil a certain quadratic equation.
Hence, there would be only finitely many possible values for these numbers.
In case, $F_{\tilde{\p}}=\Q_p$ this is obviously absurd since there are infinitely many extensions $E/F$.
In general, we cannot control the kernel of the local norm map $\No_{F_{\tilde{\p}}/\Q_p}$.
But, if $\p$ does not split in $F$, we may consider extensions of the form $E=FK$ where $K/\Q$ is an imaginary quadratic extension in which $p$ is split.
Then, the element $u\in E^{\ast}$ can be chosen to lie in $K^{\ast}$.
Hence, we have
$$\LI_p(\No_{F_{\tilde{\p}}/\Q_p}(u/\sigma(u)))=[F_{\tilde{\p}}:\Q_p]\LI_p(u/\sigma(u))$$
and we can argue as before.

In both cases we conclude that there exists an extension $E/F$ such that
$$\LI_{u/\sigma(u)}(q(\pi_E,\q))=\LI_{u/\sigma(u)}(q(\pi_B,\tilde{p})).$$
Thus, we also have $\LI_{p}(q(\pi_E,\q))=\LI_{p}(q(\pi_B,\tilde{\p}))$ by Lemma \ref{changeofdir} \eqref{change2}.
This together with equation \eqref{almost} proves the claim.
\end{proof}

\begin{Rem}
In the case $F=\Q$ Bertolini, Darmon and Iovita show the actual equality of the above $\LI$-invariants (cf.~\cite{BDI}).
They use methods introduced by Greenberg and Stevens in \cite{GS}.
More precisely, they relate both $\LI$-invariants to the logarithmic derivative of the $U_p$-eigenvalue in a Hida family passing through the given modular form.
\end{Rem}

\begin{Cor}\label{main}
Assume that hypothesis (SNV) holds.
Let $A/F$ be a $p$-ordinary modular elliptic curve, which has split multiplicative reduction at a prime $\tilde{\p}\in S_p$.
Let $q_{A,\tilde{\p}}$ be the Tate period of $A$ at $\tilde{\p}$ and $\pi$ the automorphic representation of $\PGL_2/F$ associated to $A$.
Suppose that either $[F_{\tilde{\p}}:\Q_p]=1$ holds or that $p$ does not split in $F$.
Then, the following equality holds:
$$\LI_p(\No_{F_{\tilde{\p}}/\Q_p}(q_{A,\tilde{\p}}))^{2}=\LI^{\cyc}(\pi,\tilde{\p})^{2}$$
\end{Cor}
\begin{proof}
Using similar twisting arguments as before we may replace $F$ by a totally real quadratic extension.
Therefore, there exists a Jacquet-Langlands lift of $\pi$ to a totally definite quaternion algebra and we get $\LI^{\cyc}(\pi,\tilde{\p})^{2}=\LI^{\cyc}(\pi_B,\tilde{\p})^{2}.$
By Theorem 4.10 of \cite{BG2} the automorphic period $q(\pi_B,\tilde{\p})$ is equal to the Tate period of $A$.
\end{proof}

Combining the above comparison of $\LI$-invariants with \cite{Sp}, Theorem 5.7, yields the following partial answer to Hida's exceptional zero conjecture (cf.~\cite{Hida}).
\begin{Cor}[Exceptional zero conjecture]\label{ezc}
Let $A/F$ be a $p$-ordinary modular elliptic curve and $L_p(A,s)$ its associated $p$-adic $L$-function.
Write $S_p(A)\subseteq S_p$ for the subset of primes, at which $A$ has split multiplicative reduction and $r=|S_p|$ for its cardinality.
Assume that Hypothesis (SNV) holds and that $F_\p=\Q_p$ for all primes $\p\in S_p(A)$.
Then we have
$$L_p^{(r)}(A,0)=\pm r! \prod_{\p\in S_p(A)}\hspace{-0.8em} \LI_p(q_{A,\p})\cdot \prod_{\p\in S_p - S_p(A)}\hspace{-1.3em}e_\p(A)\cdot L(A,1),$$
where $q_{A,\p}$ is the Tate period of $A$ at $\p\in S_p(A)$ and $e_\p(A)=e(\pi_\p,1)$ is the modified Euler factor at $\p \in S_p - S_p(A).$
\end{Cor}

Generalizations of Hida's conjecture to arbitrary number fields have been proposed by Deppe (cf.~\cite{De}, Conjecture 4.15) and Disegni (cf.~\cite{Disegni}).
Similarly as in Corollary \ref{ezc}, one can use the following theorem together with Bergunde's generalization of Spie\ss' theorem to arbitrary number fields (see \cite{Felix}, Theorem 4.17) to give partial answers to these conjectures.
\begin{Thm}[CM extensions]\label{CM}
Assume that Hypothesis (SNV) holds.
Let $A/F$ be a $p$-ordinary modular elliptic curve, which has split multiplicative reduction at a prime $\tilde{\p}\in S_p$ with $[F_{\tilde{\p}}:\Q_p]=1$.
Let $q_{A,\tilde{\p}}$ be the Tate period of $A$ at $\tilde{\p}$ and $\pi$ the automorphic representation of $\PGL_2/F$ associated to $A$. Then, the equality
$$\LI_p(q_{A,\tilde{\p}})=\LI^{\cyc}(\pi_E,\q)$$
holds for every totally imaginary quadratic extension $E$, in which $\tilde{\p}$ splits, and every prime $\q$ above $\tilde{\p}$.
\end{Thm}
\begin{proof}
Let $E$ be as in the theorem, $u\in E^{\ast}$ an element as defined at the beginning of Section \ref{Anti} and $\sigma$ the generator of the Galois group of $E/F$.
By similar arguments as in the proof of Theorem \ref{haupt} we know that
$$\LI_{u/\sigma(u)}(q(\pi_E,\q))=\pm \LI_{u/\sigma(u)}(q_{A,\tilde{\p}}).$$
Assume that the sign is negative.
Then, we have the chain of equalities
\begin{align*}
\LI_p(q_{A,\tilde{\p}})^{2}
&\stackrel{\ref{main}}{=}\LI^{\cyc}(\pi,\tilde{\p})^{2}\\
&\stackrel{\ref{bcc}}{=}\LI^{\cyc}(\pi_E,\q)^{2}\\
&\stackrel{}{=}\LI_p(q(\pi_E,\q))^{2}\\
&\stackrel{\ref{changeofdir}}{=}(-\LI_p(q_{A,\tilde{\p}})+2\LI_p(u/\sigma(u))^{2}.
\end{align*}
Since $\LI_p(u/\sigma(u))\neq 0$ this implies that
$$\LI_p(u/\sigma(u))=\LI_p(q_{A,\tilde{\p}}).$$
But this is absurd since $u/\sigma(u)$ is algebraic by definition while by a theorem of Barré-Sirieix, Diaz, Gramain and Philibert (cf.~\cite{LInv}) $q_{A,\tilde{\p}}$ is transcendental.
\end{proof}